\documentclass[final]{siamltex}
\pdfoutput=1
\usepackage{times}
\usepackage[utf8x]{inputenc}
\usepackage{amsmath,amsfonts,amssymb,amsopn,amscd}
\usepackage{graphicx,color,colordvi}

\graphicspath{{Figures/}}
\newtheorem{remark}{\it Remark\/}
 
\newcommand{\suml}{\sum\limits}

\usepackage{ifthen}
\newboolean{DisplaySOS}
\setboolean{DisplaySOS}{false} 

\newcommand{\eq}[1]{(\ref{#1})}
\newcommand{\SOS}[1]{\ifthenelse{\boolean{DisplaySOS}}{{\textcolor{blue}{\bf[#1]}}}{}}

 \newcommand{\Modif}[1]{\ifthenelse{\boolean{DisplaySOS}}{{\textcolor{red}{\bf{#1}}}}{#1}}

\newcounter{example}[section]

\newcommand{\matC}[2]{\left[ \begin{matrix}
#1\\
#2 
\end{matrix}\right] }

\title{A new twist for the simulation of hybrid systems using the true jump method}

\author{Romain VELTZ}
\begin{document}
\maketitle
\begin{abstract}
The use of stochastic models, in effect piecewise deterministic Markov processes (PDMP), has become increasingly popular especially for the modeling of chemical reactions and cell biophysics. Yet, exact simulation methods, for the simulation of these models in evolving environments, are limited by the need to find the next jumping time at each recursion of the algorithm. Here, we report on a new general method to find this jumping time for the \textit{True Jump Method}. It is based on an expression in terms of ordinary differential equations for which efficient numerical methods are available. As such, our new result makes it possible to study numerically stochastic models for which analytical formulas are not available thereby providing a way to approximate the state distribution for example. We conclude that the wide use of event detection schemes for the simulation of PDMPs should be strongly reconsidered. The only relevant remaining question being the \Modif{efficiency} of our method compared to the \textit{Fictitious Jump Method}, question which is strongly case dependent.
\end{abstract}


\SOS{CHV a l'air meilleure si R croit le long des trajectoires}

\SOS{Pas de vraie review de simulations}

\section{Introduction}
{T}here is a growing number of processes (from biophysics, ecology, finance, internet traffic, queuing...) that are inaccurately modeled with ODEs. A classical example is the one of chemical reactions for which the reactants are well stirred but in such a small number that a continuous description is inappropriate. A description in term of a Markov process \Modif{is more relevant in this case} \cite{gillespie_exact_1977}. 

The study of Markov models, with intrinsic noise, comes with the computation of the distribution of the states at a given time. Other methods, aimed at understanding the noisy dynamics, rely on sample paths \cite{arnold_random_1998}.
Simulation methods which provide exact sample paths are very helpful for the computation of the state distribution. There are two broad classes of simulation methods \cite{graham_stochastic_2013,cocozza-thivent_processus_1997}, the \textit{Fictitious Jump Method}\footnote{also known as rejection method} (FJM) and the \textit{True Jump Method} (TJM), see appendices. The last method was long known before being popularized by Gillespie \cite{gillespie_exact_1977,gibson_efficient_2000} such that it bears the name of the Doob-Gillespie algorithm.

There are several limitations to the above algorithms. The first occurs when the number of chemical reactions is large, the second occurs when there is a subpopulation with very fast kinetics compared to the other reactions and the third being if the \Modif{transition} rates\footnote{also called \textit{propensity functions} in the case of chemical reactions}, are time dependent (in the case of the TJM). The first two limitations lead to very long simulation times while the last one requires to accurately solve an integral equation at each iteration which can be difficult and time consuming \cite{anderson_modified_2007} (but see \cite{cocozza-thivent_processus_1997} for a FJM version).

The first two limitations can be addressed by using an approximate algorithm, for example the \textit{tau-leaping method} \cite{gillespie_approximate_2001}, or by using a time-scale separation \cite{kang_separation_2013} which provides a PDMP as limit of the suitably scaled Markov model. In short, the limit is composed of an ODE, averaging the fast component of the Markov model and a jump process representing the slow component of the original Markov model.

Finally, PDMPs \cite{davis_piecewise-deterministic_1984,davis_markov_1993} also arise in models in which one wants to couple deterministic dynamics to a chemical process, as for example in the modeling of the membrane potential of a neuron coupled with the stochastic opening/closing of the ions channels \cite{bressloff_stochastic_2014,pakdaman_fluid_2010}. The PDMPs are composed of two components: a continuous one described by the flow of an ODE and a jump process - that affects the continuous flow and describes the stochastic component of the process. Note that this class of Markov processes encompasses the one of chemical reactions with time dependent reaction rates as a subtype.

This motivates the need for the computation of \textit{exact} sample paths of PDMPs. There are here two possibilities:
\begin{enumerate}
\item \Modif{If the continuous flow is known analytically, then the FJM is an exact method. Unless the integral equation in the TJM can be solved analytically as well, the TJM is not \textit{exact} in this case and should be discarded in favor of the FJM.},
\item \Modif{If the continuous flow is not known analytically, none of both methods is exact. The error made in the simulation comes from the need to compute numerically the continuous flow, and in the case of the TJM, from the numerical computation of the jumping times. Hence, compared to the FJM, the TJM adds another error source. We shall see that our method adresses this last point.} 
\end{enumerate}

The simulation of PDMPs \cite{alfonsi_adaptive_2005,brandejsky_numerical_2010,gaujal_perfect_2008,riedler_almost_2013} is reviewed in \cite{riedler_almost_2013} where the author justifies theoretically the original method of \cite{alfonsi_adaptive_2005}, albeit focusing only on the TJM. They propose a method based on an ODE solver with event detection \cite{shampine_solving_2003} where the event is the integral equation that needs to be solved to find the jump instants (see appendix~\ref{app:true}). Such numerical methods may not be available and suffers from two majors drawbacks. First, one may need to modify a library providing an ODE solver to implement the event detection. Given the complexity of such solvers \cite{hairer_solving_1996}, this may not be trivial. The second issue with the use of event detection scheme is that the time-step is \textit{only} adapted to solve the ODE and not to locate precisely the event. This is a general word of caution concerning event detection for ODE, it is an ill-posed problem \cite{shampine_solving_2003} and may lead to excessively large numerical errors as we shall see. Finally, note that all examples in \cite{riedler_almost_2013} are non-stiff, hence being relatively easy to solve numerically. We address all these issues with one theoretical result.

\Modif{Here}, we report on a new simulation algorithm derived from the TJM. It is based on the assumption that a regular ODE solver is available without further capabilities like event detection. The precision of our method is exactly the one of the ODE solver \cite{skeel_thirteen_1986,hairer_solving_1996} making the simulation of PDMP nothing more than the simulation of an ODE. In particular, global error estimation \cite{skeel_thirteen_1986,viswanath_global_2001,cao_posteriori_2004} for ODE directly leads to global error estimation of the simulation of the PDMPs. Our new method is based on a change of time variable and recursively solving an ODE of dimension $n+1$ where $n$ is the dimension of the continuous flow of the PDMP. Hence, it comes with little additional cost given that one needs to simulate the continuous component of the PDMP anyway.

The only relevant remaining question is how our new method compares to the FJM \Modif{when the continuous flow is not known analytically}. We will come back to this \Modif{question} and provide a numerical example.

The paper is organized as follows. We first derive the theoretical result which is the basis of our method. Then we provide three numerical examples. In the first two examples, we show that the simulation methods based on event detection are to be reconsidered. We then provide a third example showing the effectiveness of our method compared to the FJM before drawing some conclusions.

\section{Description of the PDMPs}

\Modif{In this paper, we consider a PDMP by following \cite{graham_stochastic_2013} in order to simplify the description of the process. Note that a more general definition is provided in \cite{davis_piecewise-deterministic_1984,davis_markov_1993} but we stick to \cite{graham_stochastic_2013} for simplicity. A PDMP is a jump process $(X(t))_{t\in \mathbb R^+}$ where the stochastic jump instants $T_n$ are non decreasing, isolated with limit $\lim\limits_{n\to\infty}T_n=\infty$ and the dynamics between the jumps is given by a flow of homeomorphisms $(\phi^t)_{t\geq 0}$ on $\mathbb R^d$ as follows:
\begin{equation}
\left\{
\begin{array}{l}
\forall t\geq 0,\quad X(t)=\suml_{n\geq0}\phi^{t-T_n}(X({T_n}))\mathbf{1}_{\{T_n\leq t <T_{n+1}\}},\\
0=T_0<T_1\leq T_2\leq \cdots\leq \infty,\\
T_n<\infty\Rightarrow T_n>T_{n-1}\quad and\quad X({T_n})\neq X({T_n-}).
\end{array}\right.
\end{equation}}
We assume that the continuous flow is produced by an ODE: 
\begin{equation}
\left\{
\begin{array}{l}
\frac{d \phi^t(x)}{dt}=F(\phi_t(x)),\quad t\geq 0\\
\phi_{0}(x) =x.
\end{array}\right.
\label{eq:ode1}
\end{equation} 
such that $x\to F(x)$ is locally Lipschitz. \Modif{The jump $X(T_n)$ and the inter jump interval $T_n-T_{n-1}$ are independent. The law of $X(T_n)$ given $X(T_{n-1})=x$ is $\Pi(x,dy)$ and the law of the inter jump interval follows from}
\[
P\left(T_n-T_{n-1}>t|X(T_{n-1})=x\right)=exp\left(-\int_0^tR_{tot}(\phi^s(x))ds\right)
\]
\Modif{where the total transition rate function $R_{tot}$, is measurable and satisfies $R_{tot}(x)\geq 0$.} 
\Modif{We allow $$\sup_{x\in\mathbb R^d} R_{tot}(x)\leq\infty$$ but $R_{tot}$ must be locally bounded. This implies \cite{graham_stochastic_2013} that there is a finite number of jumps on any bounded time interval.}

\Modif{
\begin{remark}
Using the usual trick, we can deal with the case where the vector field $F$ in \eq{eq:ode1} is time-dependent, \textit{i.e.} $F:(x,t)\to F(x,t)$, but the existence of PDMP holds for less restrictive assumptions: $t\to F(x,t)$ can be assumed continuous rather than Lipschitz as stated above. The same trick also allows to deal with the case where $R_{tot}$ is time-dependent as well.
\end{remark}
}

\section{Theoretical result}
A general simulation algorithm, called the True Jump Method \cite{graham_stochastic_2013}, is provided in appendix~\ref{app:true}: \Modif{it is also a general method to prove existence of PDMPs \cite{graham_stochastic_2013} under the conditions stated in the previous section}.
A serious limitation to this algorithm is the need to solve the integral equation in step 3. The limitation pertains to the case\footnote{the continuous function is stationary in this case} $F=0$ if the total transition rate function $R_{tot}$ is time-dependent. In general, even if the flow $(\phi^t)$ is known analytically, the integral equation needs to be solved numerically using Newton method or the bisection algorithm for example.
Here, we propose a solution based on an ODE solver and a change of time variable.

\begin{theorem}\label{th:1}
Let us make the general assumptions:
\begin{itemize}
\item the flow $(\phi^t)_{t\geq 0}$ is globally defined,
\item $R_{tot}(x)>0$ for all $x\in\mathbb R^d$,
\item $x\to R_{tot}(x)$ \Modif{is locally Lipschitz on $\mathbb R^d$,}
\end{itemize}
Then, the solution $(X(T_n^-),T_n)$ of the integral problem in step 3. is given by $(y(S_n),\tau(S_n))$ where
\begin{equation}\left\lbrace
\begin{array}{l}
\dot y(s)= F(y(s))\ /\ R_{tot}(y(s))\\
\dot\tau(s)=1\ /\ {R_{tot}(y(s))}\\
y(0)=X(T_{n-1}),\ \tau(0)=T_{n-1}. 
\end{array}\right.
\label{eq:th}
\end{equation}
Also $X(\tau(s)) = y(s)$ for $s\in[0,S_n)$.
\end{theorem}
\begin{proof}
The trick for solving $\int_{T_{n-1}}^uR_{tot}(X(s))ds=S_n,$ in Algorithm 1., consists in the introduction of the variable $\tau(s)\geq T_{n-1}$ defined as
\[
\int_{T_{n-1}}^{\tau(s)}R_{tot}(X(t))dt=s
\]
which gives the second equation in \eqref{eq:th} by differentiating with respect to $s$. However, we need to know $X(\tau(s))$ to solve the ODE for $\tau(s)$. Writing $y(s)\stackrel{def}{=}X(\tau(s))$ and using chain rule differentiation gives the remaining equation in \eqref{eq:th}.
\end{proof}

\

The fact that the flow $\phi$ is globally defined (no ''explosions'') is a simplifying assumption of \cite{davis_piecewise-deterministic_1984}. The theorem states that solving the integral equation in step 3. of the Algorithm 1. amounts to integrate an ODE over a time interval $[0,S_n)$ where $S_n$ is randomly drawn from an exponential distribution \Modif{$\mathcal E(1)$} at each recursion of the algorithm. 

\begin{remark}
It is easy to generalize the theorem to the case where $F,R_{tot}$ are time-dependent. One finds
\begin{equation}\left\lbrace
\begin{array}{l}
\dot y(s)= F(y(s),\tau(s))\ /\ R_{tot}(y(s),\tau(s))\\
\dot\tau(s)=1\ /\ {R_{tot}(y(s),\tau(s))}
\end{array}\right.
\end{equation}
\end{remark}

\subsection{Sampling the continuous flow}
The previous method is effective for computing the jump instants and the values of the PDMP at these instants. However, if one seeks to sample the continuous part in between jumps, there is a difficulty because one only has access to $X$ through a change of time variable $\tau$. A cheap way around this issue is by adding a Poisson process $(N(t))_{t\in\mathbb R^+}$ of constant rate $\lambda>0$ to our PDMP, thereby increasing \Modif{the dimension of our system:} $d\to d+1$. This way, Theorem 1. condition $R_{tot}(x)>0$ is always satisfied because \Modif{$R_{tot}(x)\geq \lambda>0$}. This additional jump process will add jump events at rate $\lambda$, hence sampling the continuous dynamics.

\subsection{Simulations}
The simulation of the process $(X(t))$ consists in applying the Algorithm 1. together with the use of Theorem 1. We call this procedure the CHV method. The performance of the ODE solver strongly impacts the speed and ''precision'' of the simulation. It is expected that the equations \eqref{eq:th} may be stiff depending on whether a lot (resp. a few) reactions are initiated in which case the factor $1\ /\ R_{tot}$ might assume vanishing or extremely large values. Hence, we find it convenient to use the ODE solvers \cite{hindmarsh_ODEPACK_1983} for their ability to automatically switch between stiff and non-stiff methods. Note that depending on the problem, other ODE methods \cite{hairer_solving_1996} may be more appropriate than the one suggested.

\subsection{Comparison with the FJM}
As stated in the introduction, when the flow is known analytically, there is no point in choosing a quasi-exact method \textit{e.g.} the True Jump Method instead of the exact FJM\Modif{, unless of course, if the jump instants are known analytically.}

When the flow is not known analytically, both methods \Modif{FJM/CHV} are quasi-exact in that they bear a numerical error coming from integrating the ODE underlying the continuous component. The question we ask is which method is the fastest given a numerical tolerance.

For convenience, we recall the two ODEs that we need to solve between the jumps, on an interval of length $S_n\sim\mathcal E(1)$, for each method:
\begin{equation}
(FJM):\dot x = F(x)/\lambda\quad\text{ or }\quad
(CHV):\left\{
\begin{array}{l}
\dot y  =F(y)/R_{tot}(y)\\
\dot \tau = 1 /R_{tot}(y)
\end{array}\right..
\end{equation}
If $\lambda\approx R_{tot}(x)$, then (FJM) is smaller hence faster to solve. However, if $R_{tot}(x)$ varies a lot, it may be difficult to bound it effectively using a single constant $\lambda$ as required in the FJM (see appendix~\ref{app:fjm}). Hence, the choice of the method boils down to know whether it is faster (and more accurate) to solve $k$
times\footnote{for $k$ fictitious jumps} (FJM) or a single time (CHV). This depends a lot on the time stepper and the way it copes with the numerical tolerances to produce the (adaptive) time steps.

In the case of fixed time steps, one needs to see whether the cost of the larger system (CHV) compensates the cost of the generation of the $k$ random numbers.
\Modif{Hence, if the FJM does not require a lot of fictitious jumps, this method is more appropriate than CHV}.

\Modif{Finally, if the transition rate $R_{tot}$ is zero on a open set, then the CHV method breaks down and the FJM is more appropriate.} On a practical side,  the CHV method is simpler to program but the FJM assume far less regularity for $R_{tot}$.
\section{Numerical examples}
We give two examples where our new method is particularly effective compared to the standard use of event location for solving the integral equation to compute the jumping times. Then we provide an example to compare our method to the FJM. 
\subsection{Stiff problem with possibly infinite jumping times}\label{section:tcp}
The purpose of this example is to show a case where the event location in \texttt{Matlab} is actually doing a good job and compare it to our method. To make things difficult for the event location, we chose stiff dynamics between the jumps. Note that the next jumping time can be infinite in our example. We look at the following process of switching dynamics \Modif{where $X(t) = (x_c(t),x_d(t))\in\mathbb R^2$:}
\begin{equation}\left\{
\begin{array}{l}
\dot x_c = a(t)x_c\\
a(t) = -100\times \left(2(x_d(t) \mod{2}) -1\right)
\end{array}\right.
\end{equation}
starting from $x_c(0)=1$. The (unique) transition rate is $R_{tot}(x_c,x_d) = x_c$ \Modif{associated with the jump $x_d\to x_d+1$ for which $\Pi((x_c,x_d),dy) = \delta_{(x_c,x_d+1)}(dy)$}. As $x_c$ increases, the transition rate also increases. Hence, the continuous problem is stiff because one has to make sure (numerically) that $R\geq0$ despite the abrupt decrease of $x_c$. Also, the integral equation is difficult to handle numerically with high precision. The jumping time, solution of the integral equation, reads (see Theorem~\ref{th:1}):
\begin{multline}\label{eq:ex1ana}
a \stackrel{def}{=} -100 \times \left(x_d(T_{n-1})\mod{2} -1\right)\\
T_n = T_{n-1}+\frac{1}{a}\log\left(1+aS_n/x_c(T_{n-1})\right),\\ 
X(T_n) = X(T_{n-1})+aS_n.
\end{multline}
From the formula, it appears that the next jumping time can be infinite when $1+aS_n/x_c(T_{n-1})<0$ when $a<0$. The next Figure~\ref{fig:ex1} shows an example of realization where the first 15000 jumping times are computed. To compare the methods, we draw $15000$ realizations of a random variable with exponential distribution and use this sequence for the 4 examples.
\begin{figure}[htbp!]
\centering
\includegraphics[width=\linewidth]{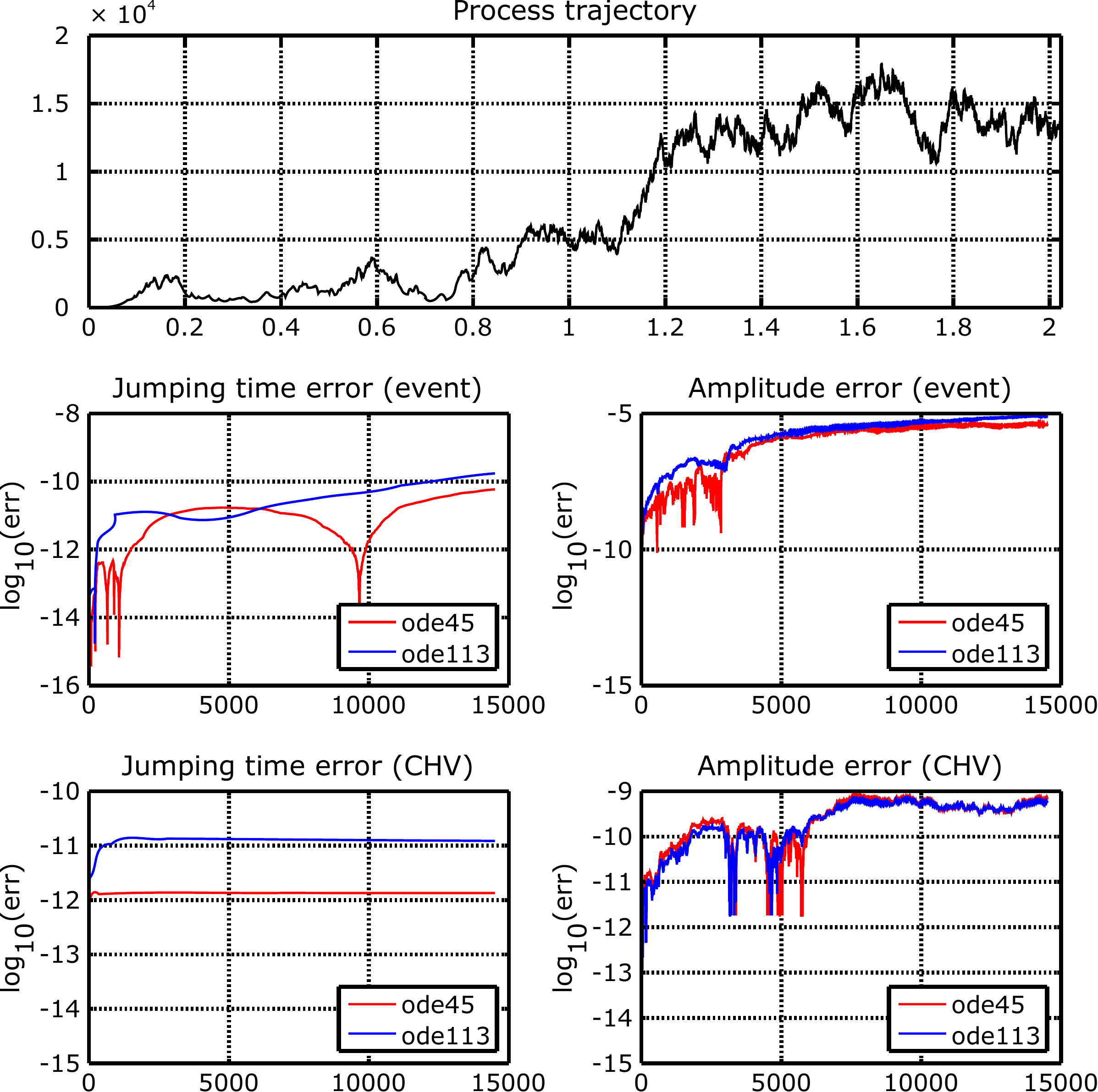}
\caption{Top: example of trajectory plotted using analytical formula \eq{eq:ex1ana}. Bottom are errors in jumping time and amplitude of the process using the event location capabilities of \texttt{Matlab} (termed event here) and using our method (termed CHV). The horizontal axis is the jump index. In both cases, the 2 smallest errors were obtained using the steppers \texttt{ode113} and \texttt{ode45}. The absolute/relative tolerance are $atol=1e-12$ and $rtol = 1e-12$.}
\label{fig:ex1}
\end{figure}

Note that we had to find a long enough trajectory because most ends quickly as the jumping time can be infinite. We see that for similar tolerances, our new method, called CHV, performs better as the number of jump events increases. Surprisingly, \texttt{ode45}, which is not designed for stiff systems performs quite well. We tried all available steppers in Matlab and only plot the two best results. 

The following table~\ref{tbl:1} provides the running times (ratios) for the cases in Figure~\ref{fig:ex1}. Despite the additional dynamical variable $\tau$ of our method, the fastest running times using either methods are almost the same. 

\begin{table}[h]
\centering
\begin{tabular}{|l|l|l|}
\hline
      & ode45 & ode113 \\ \hline
event & 1     & 1.851  \\ \hline
CHV   & 1.009 & 1.0967 \\ \hline
\end{tabular}
\caption{Computation time ratios for the examples in Figure~\ref{fig:ex1}.}
\label{tbl:1}
\end{table}

\subsection{Stiff problem with finite jumping times}
We now provide a numerical example to show that the scheme in \cite{riedler_almost_2013} can lead to wrong dynamics. We chose the following process of switching dynamics  \Modif{where $X(t) = (x_c(t),x_d(t))\in\mathbb R^2$:}
\begin{equation}\left\{
\begin{array}{lcr}
\dot x_c = 10\cdot x_c &\text{ if } & x_d \text{ even}\\
\dot x_c = -3 x_c^2 &\text{ otherwise } &
\end{array}\right.
\end{equation}
where the (unique) transition rate is $R_{tot}(x_c,x_d) = x_c$ \Modif{associated with the jump $x_d\to x_d+1$  for which $\Pi((x_c,x_d),dy) = \delta_{(x_c,x_d+1)}(dy)$}. The jumping time solution of the integral equation reads (see Theorem~\ref{th:1}):
\begin{equation}\label{eq:anaex2}
\matC{X(T_n)}{\tau(T_n)} = \left\{
\begin{array}{lcr}
\matC{10\cdot S_n+x_c(T_{n-1})}{\frac{1}{10}\log\left(1+\frac{10S_n}{x(T_{n-1})}\right)}&\text{ if } & x_d \text{ even}\\
\matC{e^{-3S_n}x(T_{n-1})}{\frac{1}{3x(T_{n-1})}\left(e^{3S_n}-1\right)}&\text{ otherwise. } & 
\end{array}\right.
\end{equation}
We see that the function $\tau$ is well-defined meaning that the jumping times are always finite. 
\begin{figure}[htbp!]
\centering
\includegraphics[width=0.99\linewidth]{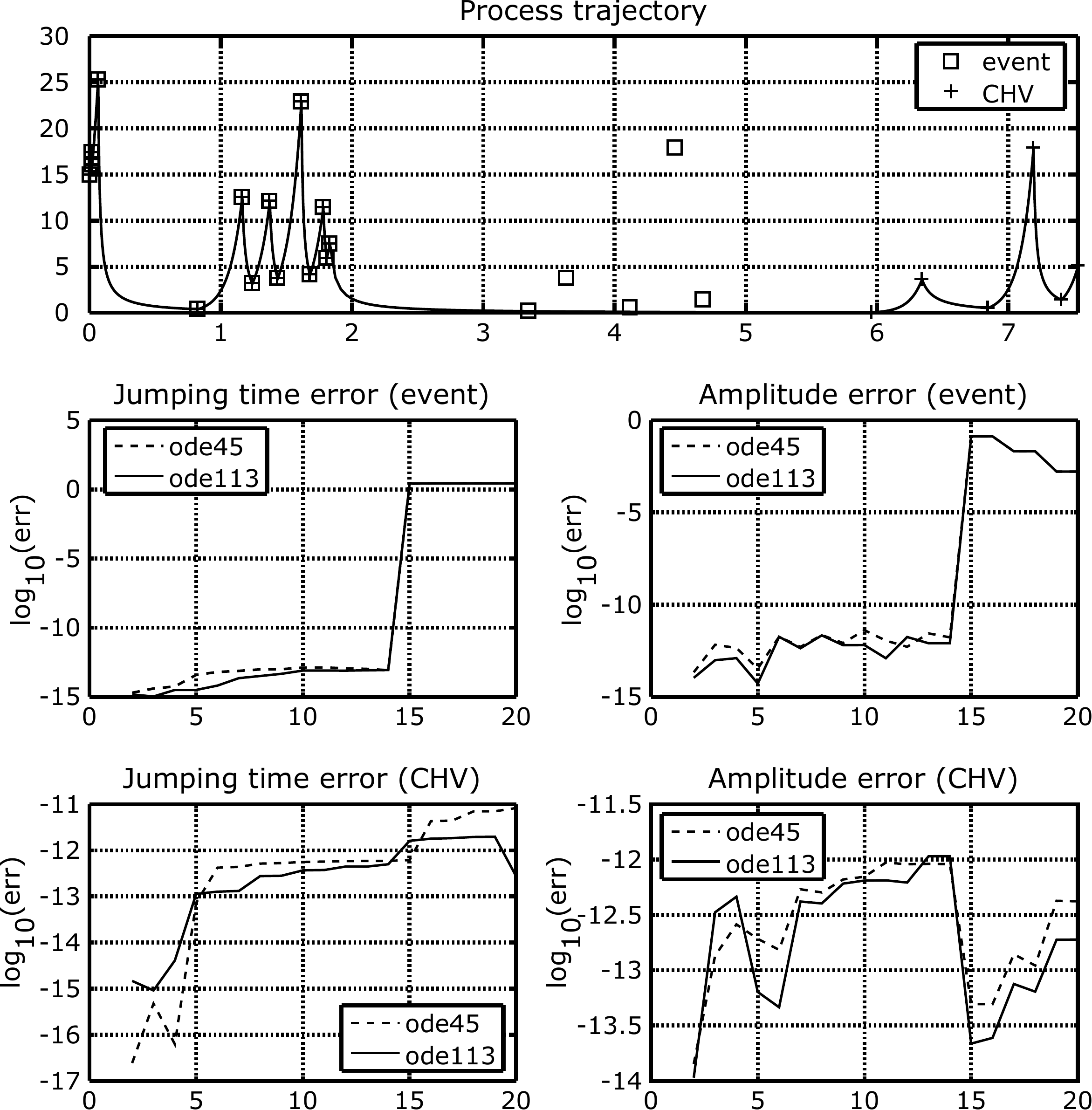}
\caption{Top: example of trajectory plotted using analytical formula \eq{eq:anaex2}. The green square are the jumps computed with the event function, the red crosses with our method. Bottom are errors in jumping time and amplitude of the process using the event location capabilities of \texttt{Matlab} (termed event here) and using our method (termed CHV). The horizontal axis is the jump index. In both cases, the 2 smallest errors were obtained using the steppers \texttt{ode113} and \texttt{ode45}. The absolute/relative tolerance are $atol=1e-12$ and $rtol = 1e-12$.}
\label{fig:ex2}
\end{figure}
The results are plotted in Figure~\ref{fig:ex2}. To compare the methods, we draw $20$ realizations of a random variable of exponential distribution and use this sequence for the 4 examples. After a few jumps (on the order of ten), the method using the \texttt{Event} function in \texttt{Matlab} fails to compute accurately a jumping time when $x_c$ is close to zero. This difficulty in computing the jumping time is also reflected in the cpu time shown in Table~\ref{tbl:ex2}. Hence, computing 100 jumps with the Event method leads to completely wrong dynamics whereas our new method gives very satisfactory results when compared to the analytical results.

\begin{table}[htbp!]
\centering
\begin{tabular}{|l|l|l|}
\hline
      & ode45 & ode113 \\ \hline
event & 4.941     & 3.569  \\ \hline
CHV   & 1.940 & 1 \\ \hline
\end{tabular}
\caption{Computation time ratios for the examples in Figure~\ref{fig:ex2}.}
\label{tbl:ex2}
\end{table}
\subsection{Comparison with the rejection method}
Finally, we provide a numerical example to compare our method to the FJM. We chose the following process of switching dynamics:
\begin{equation}\label{eq:ex3}\left\{
\begin{array}{lcr}
\dot x_c = 3\cdot x_c &\text{ if } & x_d \text{ even}\\
\dot x_c = -4 x_c &\text{ otherwise } &\\
x_c(0) = 0.05,\ x_d(0) = 0
\end{array}\right.
\end{equation}
where the (unique) transition rate is $R_{tot}(x_c,x_d) = \frac{1}{1.0 + e^{-x_c + 5.0}} + 0.1$ \Modif{with $\Pi((x_c,x_d),dy) = \delta_{(x_c,x_d+1)}(dy)$}. 
When the continuous component $x_c$ is increasing, the total rate $R_{tot}$ is bounded by $1$. However, when it is decreasing, the total rate is bounded by $R_{tot}(x_c(T_n),x_d(T_n))$ for $t\geq T_n$ where $T_n$ is the last jumping time. Hence, we expect the FJM to perform less fictitious jump when $x_d$ is odd. The fact that $x_c$ grows exponentially disfavors the FJM as we shall see.
\begin{figure}[ht!]
\centering
\includegraphics[width=0.9\linewidth]{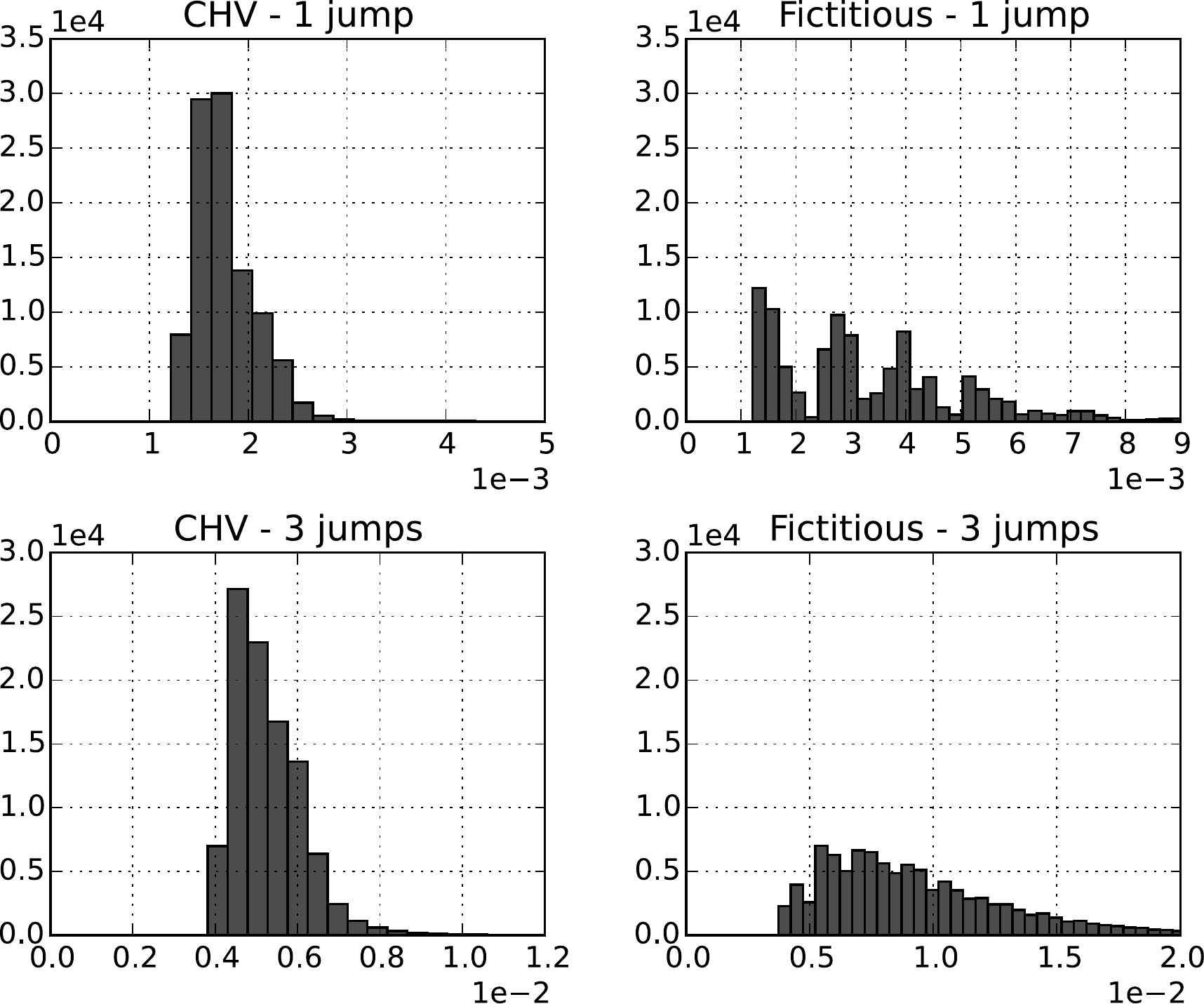}
\caption{Histograms of the computing time in seconds for each method CHV or FJM when computing 1 jump (top) and 3 jumps (bottom). The absolute/relative tolerance are $atol=1e-10$ and $rtol = 1e-10$. See text for a description of the stochastic process.}
\label{fig:ex3}
\end{figure}

To compare the two methods CHV and FJM, we use the adaptive ODE solver CVODE from SUNDIALS \cite{hindmarsh_sundials:_2005} for its ability to automatically switch between stiff and non-stiff methods. 
In each of the 4 cases, we compute 1 jump instant (resp. 3 jump instants) with each method and we perform $10^5$ realizations. We show the results in Figure~\ref{fig:ex3} where each plot is an histogram of the computing time in seconds of the jump instants. A first interesting result concerns the case \textit{FJM - 1 jump} where it seems that we can see, in the distribution, the fictitious jumps used to compute the first jump. Indeed, in our implementation, it takes roughly $1.5ms$ to solve the ODE with both methods. The top two plots show that the CHV method is faster for the reasons stated at the beginning of the section. The other 2 plots show that this result is not affected by the computation of the jumping time in the case $x_d$ odd, where the FJM should be slightly faster. Hence, the fastest method for simulating the process \eq{eq:ex3} would be to use our method when $x_d$ is even and the rejection method otherwise.

\section{Generalizations}
Several possible generalizations of our model suggest themselves. Basically, these entail a modification of the flow \eqref{eq:ode1} or of the jump process.

If the flow can be written as a Cauchy problem, then our method still applies albeit possibly in a different state space. For example, we can consider the case of a flow generated by a delay differential equation \cite{hale_introduction_1993} or if jump process depends on the history of the process $(X(t))$. These cases are handled very similarly to the case presented here, except that \eqref{eq:ode1}, \eqref{eq:th} are now delay differential equations. One could also look at the case of a flow generated by partial differential differential equations \cite{buckwar_exact_2011}.

The next possible generalization was introduced in
\cite{bratsun_delay-induced_2005} where a time delay was added between the initiation and the completion of some, or all, of the reactions. These delays were motivated by oscillations observed in gene regulation networks \cite{denault_wc2_2001}. A modified True Jump Method was given \cite{bratsun_delay-induced_2005} where a table is added to record when the delayed reactions are scheduled. The algorithm was later improved by \cite{cai_exact_2007,anderson_modified_2007} but the method is the same. Hence, the computation of the next reaction time follows from the same integral equation that we solved here and our results equally apply. 


\section{Discussion}
We have identified a new way to simulate a large class of Markov processes, the PDMPs, with a method which is versatile enough to be applied to various continuous part dynamics (ODE / PDE / delay differential equations) and any continuous \Modif{rate function $R_{tot}$}. Indeed the precision of the method is directly linked to the one of the ODE solver, which we assume is arbitrary given the state of the art concerning global error estimation \cite{skeel_thirteen_1986,viswanath_global_2001}.

As a particular case, our method provides a solution to cope with the notoriously difficult problem \cite{anderson_modified_2007} of the simulation of chemical reactions with time dependent propensity functions. Nevertheless, this problem is only formal as only the Fictitious Jump Method is exact in this case.

\Modif{Our result is an} important achievement for the following reasons. First, there is no need to modify existing ODE solvers to implement \Modif{our method}. Second, \Modif{it} provides a way to use an adaptive discretization to approximate \textit{both} the ODE flow and the integral equation. As such, \Modif{it} makes the method \cite{alfonsi_adaptive_2005,riedler_almost_2013} for the simulation of PDMPs not relevant. \Modif{It also allows to focus on the choice of the ODE solver without} \textit{ad-hoc} modification of existing tools, this is especially useful when studying stiff systems such as the slow-fast ones, largely present in the modeling of biology. Our results also makes it easy to consider flows driven by PDE or delay differential equations for example, without sacrificing for the numerical precision. Further, our results bridge the gap between the numerical methods for the simulation of PDMPs and the ones of ODE, making the first a subproblem of the second class. In particular, numerical error analysis \cite{hairer_solving_1996,skeel_thirteen_1986,viswanath_global_2001,cao_posteriori_2004} of ODE directly maps onto PDMPs.

The only relevant question which remains is whether to chose our method or the FJM when the continuous flow is not known analytically and when the total rate function $R_{tot}$ is continuous. This is strongly dependent on the total rate function variations and not so much on its absolute values. If the probability of having at least 2 fictitious jumps is 'high', then our new method might be the most efficient while being simpler to program. Finally, note that mixing the two methods can be effective by choosing the fastest one depending on the value of the discrete component as we suggested in our third example.

\section*{Acknowledgment}
This work was partially supported by the European Union Seventh
Framework Programme (FP7/2007-2013)
under grant agreement no. 269921 (BrainScaleS), no. 318723 (Mathemacs),
and by the ERC advanced grant
NerVi no. 227747 and by the Human Brain Project (HBP).

\SOS{Faugeras Taley Benoite Tanre?}
\appendix
\section{True Jump Method}\label{app:true}
We recall the true jump method \cite{graham_stochastic_2013} for the exact simulation of $(X(t))_{t\in\mathbb R^+}$ on a time interval $[0,T]$. This is the SSA Doob-Gillespie algorithm \cite{gillespie_exact_1977} when $F=0$, for which $\phi^t=Id$.

\textbf{Algorithm 1.} (True Jump Method)
\begin{enumerate}
\item Draw the initial value $X_0$
\item Draw $S_n$ according to an exponential distribution $\mathcal E(1)$
\item Set $T_n=\inf\left\lbrace u>T_{n-1}\ :\ \int_{T_{n-1}}^uR_{tot}(\phi^s(X(T_{n-1})))ds=S_n \right\rbrace$
and $X(t)=\phi^t(X(T_{n-1}))$ for $T_{n-1}\leq t<T_n$.
\item Draw $X(T_n)$ according to the law $\Pi(\phi^{T_n-T_{n-1}}(X(T_{n-1})),dy)$. 
\item If $T_n<T$, return to step 2.
\end{enumerate}

\section{Fictitious Jump Method}\label{app:fjm}
We recall the "rejection" method \cite{graham_stochastic_2013} for the exact simulation of $(X(t))_{t\in\mathbb R^+}$ on a time interval $[0,T]$. We assume that $\sup\limits_{x\in\mathbb R^d}R_{tot}(x)\leq\lambda<\infty$.

\textbf{Algorithm 2.} (Fictitious Jump Method)
\begin{enumerate}
\item Draw the initial value $X_0$
\item Draw $S_n$ according to an exponential distribution $\mathcal E(\lambda)$
\item Set $X(t)=\phi^t(X(T_{n-1}))$ for $T_{n-1}\leq t<T_{n-1}+S_n$.
\item
\begin{itemize}
\item either, with probability $1-\frac{R_{tot}(X(T_n-)))}{\lambda}$, set $X(T_n) = X(T_n-)$
\item or else draw $X(T_n)$ according to the law $\Pi(X(T_n-),dy)$
\end{itemize} 
\item If $T_n<T$, return to step 2.
\end{enumerate}

\appendix

\bibliographystyle{alpha}
\bibliography{/Users/rveltz/Dropbox/zotero}

\end{document}